\documentclass[1p]{elsarticle}
\usepackage{amssymb}
\usepackage{amsfonts}
\usepackage{amsmath}
\usepackage{amsthm}

\usepackage[polish, english]{babel}
\usepackage{polski}

\newtheorem{theorem}{Theorem}
\newtheorem{conjecture}[theorem]{Conjecture}
\newtheorem{corollary}[theorem]{Corollary}
\newtheorem{lemma}[theorem]{Lemma}
\newtheorem{observation}[theorem]{Observation}
\newtheorem{problem}[theorem]{Problem}

\newtheorem{claim}{Claim}

\begin{document}

\title{On generalised majority edge-colourings of graphs}

\author[agh]{Pawe{\l} P\k{e}ka{\l}a}
\ead{ppekala@agh.edu.pl}

\author[agh]{Jakub Przyby{\l}o}
\ead{jakubprz@agh.edu.pl}

\address[agh]{AGH University, al. A. Mickiewicza 30, 30-059 Krakow, Poland}

\begin{abstract}
A \textit{$\frac{1}{k}$-majority $l$-edge-colouring} of a graph $G$ is a colouring of its edges with $l$ colours such that for every colour $i$ and each vertex $v$ of $G$, at most $\frac{1}{k}$'th of the edges incident with $v$ have colour $i$. We conjecture that for every integer $k\geq 2$, each graph with minimum degree $\delta\geq k^2$ is $\frac{1}{k}$-majority $(k+1)$-edge-colourable and observe that such result would be best possible. This was already known to hold for $k=2$. We support the conjecture by proving it with $2k^2$ instead of $k^2$, which confirms the right order of magnitude of the conjectured optimal lower bound for $\delta$. 
We at the same time improve the previously known bound of order $k^3\log k$, based on a straightforward probabilistic approach.
As this technique seems not applicable towards any further improvement, we use a more direct non-random approach.
We also strengthen our result, 
in particular substituting $2k^2$  by $(\frac{7}{4}+o(1))k^2$. Finally, we provide the proof of the conjecture itself for $k\leq 4$ and completely solve an analogous problem for the family of bipartite graphs. 
\end{abstract}


\maketitle

	\section{Introduction}

	A \textit{majority colouring} of a digraph $D$ is a colouring of the vertices of $D$ such that for every vertex of $D$ at most half the out-neighbours of $v$ receive the same colour as $v$. 
  This notion was first considered by Kreutzer, Oum, Seymour, van der Zypen and Wood~\cite{digraph1}, who in particular proved that every digraph has a majority $4$-colouring, and conjectured that $3$ colours should always suffice. They also posed several other related problems, addressed in a few consecutive papers.
  In particular, in~\cite{digraph2} Anholcer, Bosek and Grytczuk extended the result above to a list setting. 
  Further, in~\cite{digraph3,Knox-Samal} the authors
  studied the problem of $\frac{1}{k}$-majority colourings of digraphs, that is such colourings of the vertices of a digraph that each vertex receives the same colour as at most $\frac{1}{k}$'th of its out-neighbours, which is a natural generalisation, proposed already in~\cite{digraph1}.
  Girão, Kittipassorn and Popielarz~\cite{digraph3}, 
  and independently 
Knox and \v{S}\'amal~\cite{Knox-Samal} proved that for each $k\geq 2$, every digraph is $\frac{1}{k}$-majority $2k$-colourable,
while there are digraphs requiring no less than $2k-1$ colours.
  Further results and extensions can be found e.g. in~\cite{Szabo-majority,digraph2,MajorityGeneralOur}.

It is worth mentioning that optimal results concerning the natural correspondents of the notions above, but in the environment of graphs follow by the argument of Lov\'asz~\cite{Lovasz-majority}, printed already in 1966, see also~\cite{MajorityGeneralOur} for further comments and results.

	A \textit{majority edge-colouring} of a graph $G$ is a colouring of the edges of $G$ such that for each vertex $v$ of $G$, at most half of the edges incident with $v$ have the same colour. More generally, for an integer $k\geq 2$, a \textit{$\frac{1}{k}$-majority edge-colouring} of $G$ is a colouring of its edges such that for every colour $i$ and each vertex $v$ of $G$ at most $\frac{1}{k}$'th of the edges incident with $v$ have colour $i$. One of characteristic features of these notions, introduced recently by Bock, Kalinowski, Pardey, Pil\'sniak, Rautenbach and Wo\'zniak~\cite{majority23},
 is that unlike in the case of vertex-colourings, such edge-colourings do not exist for all graphs. 
 In particular, for every $k\geq 2$, graphs with vertices of degree $1$ do not admit a $\frac{1}{k}$-majority edge-colouring with any number of colours.  
 In~\cite{majority23} it was however proven that every graph $G$ of minimum degree $\delta\geq 2$ has a majority $4$-edge-colouring.
 On the other hand, the minimum degree of a graph may have significant influence on the number of colours sufficient to provide such colourings, and examining this problem 
 seems to be the primal issue in this area.
 The main result of~\cite{majority23} solves this problem for $k=2$.
 \begin{theorem}[\cite{majority23}]\label{majority23-res2}
Every graph $G$ of minimum degree $\delta\geq 4$ has a majority $3$-edge-colouring.
 \end{theorem}
This result is twofold best possible. Firstly, $4$ cannot be decreased, as exemplify e.g. cubic graphs with chromatic index $4$.
Secondly, $2$ colours are not sufficient e.g. for any graph containing an odd degree vertex.

The main motivation of our research is thus the quest for best possible extension of Theorem~\ref{majority23-res2} towards all $k\geq 3$.
Note first that for any fixed $k\geq 2$, no minimum degree constraint can guarantee the existence of a $\frac{1}{k}$-majority edge-colouring with at most $k$ colours -- it is enough to consider graphs containing vertices of (arbitrarily large) degrees not divisible by $k$. We thus must admit (at least) $k+1$ colours, and in fact the authors of~\cite{majority23} showed that this number of colours is sufficient (and hence, optimal) within our quest. 
 \begin{theorem}[\cite{majority23}]
 For every integer $k\geq 2$, there exists $\delta_k$ such that each
graph $G$ of minimum degree $\delta\geq \delta_k$ has a $\frac{1}{k}$-majority $(k+1)$-edge-colouring.
 \end{theorem}
For $k=2$ this follows by Theorem~\ref{majority23-res2}, while in the remaining cases it was proven by a fairly standard application of the Lov\'asz Local Lemma, which admitted to get the result above with $\delta_k=\Omega(k^3 \log k)$. We believe that much smaller values of $\delta_k$ should allow 
obtaining the same result. Our main objective concerns finding the optimal value of $\delta_k$ and can be formulated as follows.
\begin{problem}\label{GeneralizedMajorityProblem}
For every integer $k\geq 2$, find the least value $\delta_k^{\rm opt}$ such that each
graph $G$ of minimum degree $\delta\geq \delta_k^{\rm opt}$ has a $\frac{1}{k}$-majority $(k+1)$-edge-colouring.
\end{problem}
It seems one needs to use a non-probabilistic argument in order to completely solve this problem, we discuss this issue at the end of the paper. 
In the next section we present the main tool we shall use within our approach, and show how it allows to solve the problem in the environment of bipartite graphs. In Section~\ref{SectionGeneralGraphs} we shall 
in turn provide order-wise tight estimations for $\delta_k^{\rm opt}$ 
in the general case and formulate our main conjecture.
In the following section we also confirm that the conjecture holds for $k\leq 4$. The last section is devoted to a short discussion concerning our results and further perspectives.

\section{Bipartite graphs}

Let us begin with a simple observation implying the lower bound for $\delta_k^{\rm opt}$, also within the family of bipartite graphs.

\begin{observation}\label{ObservationLowerBipartite}
    For every $k\geq 2$ there exist bipartite graphs $G$ with $\delta(G)\geq k^2-k-1$ which are not $\frac{1}{k}$-majority $(k+1)$-edge-colourable. 
\end{observation}

\begin{proof}
    Let $G$ be a graph containing a vertex $v$ of degree $k^2-k-1$ and suppose $G$ has a $\frac{1}{k}$-majority $(k+1)$-edge-colouring. 
    Then at most 
    $$\left\lfloor\frac{k^2-k-1}{k}\right\rfloor=k-2$$ edges incident with $v$ may have the same colour, and hence at most
	\[ (k+1)(k-2) = k^2-k-2 < d(v) \]
	edges incident with $v$ can be coloured, a contradiction.

Therefore, in particular the complete bipartite graph $K_{k^2-k,k^2-k}$ with a single edge removed is an example of a bipartite graph which is not $\frac{1}{k}$-majority $(k+1)$-edge-colourable and has minimum degree $k^2-k-1$.
 \end{proof}

 We thus must require that $\delta(G)\geq k(k-1)$ in order to have a chance to show that such assumption guarantees that $G$ is $\frac{1}{k}$-majority edge-colourable with $k+1$ colours. In~\cite{majority23} it was actually already proven that $k+2$ colours are sufficient in such a case.
 We shall show that in fact in the case of bipartite graphs, $k+1$ colours always suffice, which, in view of Observation~\ref{ObservationLowerBipartite}, yields an optimal result. 
 	\begin{theorem} \label{bipartiteTheorem}
		For every integer $k \geq 2$, if a bipartite graph $G$ has minimum degree at least $k(k-1)$, then $G$ has a $\frac{1}{k}$-majority $(k+1)$-edge-colouring.
	\end{theorem}
In order to prove Theorem~\ref{bipartiteTheorem} we shall make use of Lemma~\ref{mod_alon} below. 
 This was essentially proven by Alon and Wei~\cite{alon}.
We present a slight, yet very useful refinement of their result.
We also include its full proof for the sake of completeness.
We say a cycle is \emph{odd} if it has odd length.
Moreover, cycles $C_1,\ldots,C_t$ in a graph $G$ are called \emph{independent} if for any $i\neq j$ there is no vertex of $C_i$ adjacent to a vertex of $C_j$ in $G$. (Note that such cycles are in particular pairwise distinct.)
    \begin{lemma} \label{mod_alon}
		Let $G = (V,E)$ be a graph, and let $z:E\rightarrow [0,1]$ be a weight function assigning to each edge $e\in E$ a real weight $z(e)$ in $[0,1]$. Then there is a function $x: E\rightarrow \{0,1\}$ assigning to each edge an integer value in $\{0,1\}$ such that 
\begin{itemize}
    \item[(i)] $\sum\limits_{e\ni v} z(e) -1 < \sum\limits_{e\ni v} x(e) \leq \sum\limits_{e\ni v} z(e) +1$ for every $v\in V$;
    \item[(ii)] if $\sum\limits_{e\ni u} x(e) < \sum\limits_{e\ni u} z(e)$ and $\sum\limits_{e\ni v} x(e) < \sum\limits_{e\ni v} z(e)$ for some $uv\in E$, then $x(uv)=1$;
    \item[(iii)] each vertex $v$ with $\sum\limits_{e\ni v} x(e)=\sum\limits_{e\ni v} z(e)+1$ belongs to an odd cycle $C_v$ with $\sum\limits_{e\ni u} z(e)\in\mathbb{Z}$ for every vertex $u$ of $C_v$, and moreover all such cycles are independent in $G$.
\end{itemize}
	\end{lemma}
 \begin{proof} 
    If $z(e) \in \{0,1\}$ for any edge $e$ of $G$, then fix $x(e) = z(e)$. Let $G'$ be a subgraph of $G$ obtained by removing all edges of $G$ for which $z(e)$ is an integer. Consider an incidence matrix $M$ of $G'$. For every edge $e\in E(G')$ let $y_e$ be its corresponding column in $M$. Note that if $G'$ contains a closed walk of even length whose every two consecutive edges are distinct from each other, then the columns of $M$ are linearly dependent over reals, i.e. there exist real numbers 
    $\alpha_e$, $e\in E(G')$, such that
    \[ \sum_{e\in E(G')} \alpha_e y_e = \overline{0} \]
    and at least one $\alpha_e$ 
    is nonzero. 

    Note that for any nonzero real number $c$, if we modify all of the values $x(e)$ by $c\alpha_e$, then the sum $\sum_{e\ni v} x(e)$ remains the same for all vertices $v$, but the value of $x(e)$ shall change for at least one edge $e$. Therefore we can choose the value of the coefficient $c$ so that all modified values of $x(e)$ remain in $[0,1]$, but at least one of them gets an integer. Remove all integer valued edges from the graph $G'$ and continue with the same procedure until $G'$ does not contain an even closed walk with no repeated consecutive edges. Note every component of the resulting graph $G'$ contains
    at most one cycle, which has to be odd.

    Observe that at this point for all vertices $v$ of $G$ we have
    \[ \sum_{e\ni v} x(e) = \sum_{e\ni v} z(e) \]
    and for all edges $e$ of $G'$ the value of $x(e)$ is in the open interval $(0,1)$. Hence, after modifying each of these values to an integer in $\{0,1\}$,
    for all vertices $v$ with $d_{G'}(v)\leq 1$ we shall have
    \[ \sum\limits_{e\ni v} z(e) -1 < \sum\limits_{e\ni v} x(e) < \sum\limits_{e\ni v} z(e) +1 \text{.} \]
    Thus, in the following step we shall focus on modifying the values of $x(e)$ for the edges $e$ of $G'$ in such a way that 
    \[ \sum_{e\ni v} x(e) = \sum_{e\ni v} z(e) \]
    for all vertices $v$ with $d_{G'}(v)\geq 2$.

    Suppose $G'$ has a component containing both a vertex of degree $1$ and a vertex of degree at least $2$. 
    Consider the system of linear equations
    \[ \sum_{e\ni v} x(e) = \sum_{e\ni v} z(e) \]
    for all vertices $v$ of the component which have degrees at least $2$, where the variables are $x(e)$ for all edges $e$ of this component. The number of variables in this system is greater than the number of equations, hence its solution set is infinite. 
    We choose a solution such that all values of $x(e)$ remain in the interval $[0,1]$ and at least one of them is an integer. We then remove integer valued edges from $G'$. We repeat this procedure until all components of $G'$ are odd cycles or isolated edges. 
    
    For all isolated edges $e$ of $G'$ we can then simply set $x(e)=1$.

    The remaining edges $e$ with non-integer values of $x(e)$ induce disjoint odd cycles in $G$. By previous arguments all vertices $v$ of such cycles satisfy
    \[ \sum_{e\ni v} x(e) = \sum_{e\ni v} z(e) \text{.} \]

    We call an odd cycle \emph{bad} if $x(e)=1/2$ for all its edges $e$. 
    Note $\sum_{e\ni v} z(e)\in \mathbb{Z}$ for every vertex $v$ of any bad cycle in $G'$.
    We shall show that we may modify $G'$ so that all its bad cycles are independent in $G$. Suppose there are bad (odd) cycles $C_1,C_2$ in $G'$ joined by an edge, say $e_0$ in $G$. Let $H$ be the subgraph of $G$ induced by $e_0$ and the edges of $C_1,C_2$. Then it is straightforward to notice that there exist $\alpha_e\in\{-1,1\}$, $e\in C_1\cup C_2$, such that
    \[ \alpha_{e_0}y_{e_0} + \sum_{e\in C_1\cup C_2}\alpha_ey_e = \overline{0} \]
    where $\alpha_{e_0}=2$.
    (It suffices to alternately set $\alpha_e$ to $-1$ and $1$ along both of the cycles, starting from a vertex of $e_0$.) Since $x(e)=1/2$ for $e\in C_1\cup C_2$ and $x(e_0)\in\{0,1\}$, by adding $\alpha_e/2$ or $-\alpha_e/2$ (depending on $x(e_0)$) to all $x(e)$, $e\in E(H)$ we shall thus not change the sum $\sum_{e\ni v} x(e)$ at any vertex $v$ while all edges of $H$ shall get integer valued. We shall thus remove all these edges from $G'$. We repeat this procedure until all bad cycles of $G'$
    are independent in $G$.

    For each cycle $C$ of $G'$ we then proceed as follows. 
    If $C$ is bad, we denote any of its vertices as $v$.
    Then for all edges $e$ of $C$ we round the value of $x(e)$ to the nearest integer with the additional restriction that if for the two edges $e',e''$ incident to any given vertex $u$ 
    in $C$ we had $x(e') = 1/2 = x(e'')$, then one of these values must be rounded down to $0$ and the other one rounded up to $1$ if $u\neq v$, while we round both values up to $1$ for $u=v$.
    As a result we obtain a function $x$ such that $\sum\limits_{e\ni v} x(e) = \sum\limits_{e\ni v} z(e) +1$ 
    and 
    \[ \sum\limits_{e\ni u} z(e) -1 < \sum\limits_{e\ni u} x(e) < \sum\limits_{e\ni u} z(e) +1 \]
    for all the remaining vertices $u$ of $C$ (other than $v$). Thus, since bad cycles of $G'$ were independent in $G$, we obtain a function $x: E\rightarrow \{0,1\}$ satisfying (i) and (iii).

    Finally, we shall show that we can further modify the function $x$ so that it also satisfies (ii). Assume (ii) is not satisfied and there exists an edge $uv \in E$ with $x(uv)=0$ such that $\sum\limits_{e\ni u} x(e) < \sum\limits_{e\ni u} z(e)$ and $\sum\limits_{e\ni v} x(e) < \sum\limits_{e\ni v} z(e)$. If we modify the value of $x(uv)$ to $1$, then $\sum\limits_{e\ni u} x(e) < \sum\limits_{e\ni u} z(e) +1$ and $\sum\limits_{e\ni v} x(e) < \sum\limits_{e\ni v} z(e) +1$, hence (i) and (iii) are still satisfied, but there are less edges that contradict (ii). Therefore, we can construct in this manner a function $x: E\rightarrow \{0,1\}$ satisfying all three conditions (i) -- (iii).
 \end{proof}

	\begin{proof} [Proof of Theorem \ref{bipartiteTheorem}]
		Let $G=(V,E)$ be a bipartite graph with $\delta(G)\geq k(k-1)$, $k\geq 2$. 
  
  Set $\overline{G}_{k+1} = G$. 
  
  For $i=k+1,k,\ldots,2$ let further $G_i$ be a subgraph of $\overline{G}_{i}$ obtained via applying to it Lemma \ref{mod_alon} with a constant weight function $z_i(e)=\frac{1}{i}$
  and setting $E(G_i)=\{e\in E(\overline{G}_{i}): x_i(e)=1\}$  
   where $x_i:E(\overline{G}_{i})\to\{0,1\}$ is a function resulting from the lemma; 
  let also $\overline{G}_{i-1} = (V(G), E(\overline{G}_i) \setminus E(G_i))$.

  Finally, set $G_1 = \overline{G}_1$.
   
  We shall prove that for every $i \in \{1,\dotsc,k+1\}$ and each vertex $v\in V$
		\begin{equation} \label{BipartiteDegreIneq}
                d_{G_i}(v) \leq \left\lfloor \frac{d_G(v)}{k} \right\rfloor, 
            \end{equation}
  and thus the edges of $G_1,\ldots,G_{k+1}$ partition $E$ to $k+1$ colour classes inducing a  
  $\frac{1}{k}$-majority $(k+1)$-edge-colouring of $G$.
  In other words, the colouring $c:E\to\{1,\ldots,k+1\}$ can be defined by setting $c(e)=i$ iff $e\in E(G_i)$.
		
		Let $v$ be an arbitrary vertex of $G$ and let 
  $$d_G(v) = (k+1) l + j$$ 
  where $j\in \{0,\dotsc,k\}$. Since $d_G(v) \geq \delta(G) \geq k(k-1)$ we have that $l\geq k-1$ or $l=k-2$ and $j\geq 2$. Hence
		\begin{equation} \label{l+1-Estimation}
  \left\lfloor \frac{d_G(v)}{k} \right\rfloor = \left\lfloor \frac{(k+1) l + j}{k} \right\rfloor = l+ \left\lfloor \frac{l + j}{k} \right\rfloor \geq l+1
  \end{equation}
		unless $l=k-1$ and $j=0$. 
  Since $G$ is bipartite, none of $G_i$ contains odd cycles. Thus,
  Lemma~\ref{mod_alon} (iii), exploited to construct each $G_i$, guarantees the following to hold.
  \begin{claim}\label{Claim1-bipartite}
If $d_{\overline{G}_{t}}(v) = td$ for some $t\in\{1,\ldots,k+1\}$ and $d\in\mathbb{Z}$, hence $\sum\limits_{e\ni v} z_t(e)=d$, then $d_{G_t}(v) = d$.
  \end{claim}
 This almost immediately yields~(\ref{BipartiteDegreIneq}) in the case when $d_G(v)$ is divisible by $k+1$. It is sufficient to apply $k$ times Claim~\ref{Claim1-bipartite} to obtain the following.
 \begin{claim}\label{Claim2-bipartite}
     If $j=0$, i.e. $d_G(v) = (k+1)l$, then $d_{G_i}(v) = l$ for all $i$.
 \end{claim}
 Then however, we have that for every $i \in \{1,\dotsc,k+1\}$,
		\[ d_{G_i}(v) = l = \left\lfloor \frac{kl}{k} \right\rfloor \leq \left\lfloor \frac{d_G(v)}{k} \right\rfloor \text{,} \]
  and thus~(\ref{BipartiteDegreIneq}) follows.

		 It remains to prove~(\ref{BipartiteDegreIneq}) in the case when $j\neq 0$. By~(\ref{l+1-Estimation}) it is sufficient to show that $d_{G_i}(v)\leq l+1$ for every $i$. This is implied by $k$ times repeated application of the following claim.
\begin{claim}
  If $d_G(v) = (k+1)l + j$ and $j\neq 0$,   
  then for every $t \in \{2,\dotsc,k+1\}$, if $d_{\overline{G}_t}(v) \in \{ tl, tl+1, \dotsc, t(l+1) \}$, then $d_{G_t}(v) \in \{l,l+1\}$ and $d_{\overline{G}_{t-1}}(v) \in \{(t-1)l,(t-1)l+1,\dotsc,(t-1)(l+1)\}$.
\end{claim}   
\begin{proof}
 Note that analogously as above, if $d_{\overline{G}_t}(v) = tl$ (respectively, $t(l+1)$), then by Claim~\ref{Claim1-bipartite}, $d_{G_t}(v) = l$ (respectively, $l+1$) and $d_{\overline{G}_{t-1}}(v) = (t-1)l$ (respectively, $(t-1)(l+1)$). 
 
 In the remaining cases, $d_{G_t}(v) \in \{l,l+1\}$ by Lemma~\ref{mod_alon} (i), and thus  $d_{\overline{G}_{t-1}}(v) \in \{(t-1)l,(t-1)l+1,\dotsc,(t-1)(l+1)\}$. 
 \end{proof}
	\end{proof}

\section{General graphs}\label{SectionGeneralGraphs}

The main obstacle in obtaining a similar result as in Theorem~\ref{bipartiteTheorem} for the general case, not only for bipartite graphs, is that we cannot show Claim~\ref{Claim1-bipartite} to hold any more if a graph has odd cycles, as the proof of this fact relied on Lemma~\ref{mod_alon} (iii). Actually, this inconvenience shall have much further reaching consequences than we initially suspected, and shall (most likely) disallow us obtaining sharp results in most of the general cases.

Before we discuss our upper bounds, let us however first demonstrate that it is after all not that surprising we could not solve this apparent sole obstacle on the way towards extending Theorem~\ref{bipartiteTheorem} to all graphs, as the upper bound in it does not hold any more in general.

\begin{observation}\label{ExampleGeneral}
    For every $k\geq 2$ there exists a graph $G$ with $\delta(G)\geq k^2-1$ which is not $\frac{1}{k}$-majority $(k+1)$-edge-colourable. 
\end{observation}

\begin{proof}
    For every fixed $k\geq 2$ we construct a graph $G$ as follows.
    We first take 
    the complete graph on $k^2+1$ vertices 
    and remove the edges of any fixed Hamilton cycle from it.
    Then we add a new vertex to the constructed graph and connect it with single edges with all the remaining vertices. Note that the obtained graph $G$ has $k^2+1$ vertices of degree $k^2-1$ and a single vertex of degree $k^2+1$. Suppose there is a $\frac{1}{k}$-majority $(k+1)$-edge-colouring of $G$. Since $k^2-1=(k+1)(k-1)$ and at most 
    $$\left\lfloor\frac{k^2-1}{k}\right\rfloor = k-1$$ edges incident to every vertex $v$ of degree $k^2-1$ can be coloured with the same colour, then for every such vertex $v$ and for each of the $k+1$ colours, exactly $k-1$ edges incident to such $v$ are coloured with this colour. Analogously, for the only vertex $u$ of $G$ with degree $k^2+1 = (k+1)(k-1)+2$,
    at most 
     $$\left\lfloor\frac{k^2+1}{k}\right\rfloor = k$$ 
     edges incident to the vertex $u$ can be coloured with the same colour, and hence exactly $k$ edges incident to the vertex $u$ must be coloured with some colour, say $\alpha$. Consider the subgraph of $G$ induced by the edges of $G$ coloured with $\alpha$. The sum of degrees in this subgraph equals 
     $$(k^2+1)(k-1)+k,$$ 
     which is always an odd number, a contradiction. 
\end{proof}

Observation~\ref{ExampleGeneral} thus implies that $\delta_k^{\rm opt}\geq k^2$ for every $k\geq 2$. We conjecture that in fact $\delta_k^{\rm opt} = k^2$ for each $k\geq 2$, which holds for $k=2$ by Theorem~\ref{majority23-res2}.

	\begin{conjecture}\label{EdgeMajconj}
		For every integer $k\geq 2$, if a graph $G$ has minimum degree $\delta\geq k^2$, then $G$ is $\frac{1}{k}$-majority $(k+1)$-edge-colourable.
	\end{conjecture}

By means of Lemma~\ref{mod_alon} we shall now show that $\delta_k^{\rm opt}$ is at most twice the conjectured value. The proof of this fact shall also be an indispensable ingredient of the further slight improvement included in Theorem~\ref{thm2-GenMajEdg}.

	\begin{theorem} \label{thm1-GenMajEdg}
		For every integer $k\geq 2$, if a graph $G$ has minimum degree $\delta\geq 2k^2$, then $G$ is $\frac{1}{k}$-majority $(k+1)$-edge-colourable.
	\end{theorem}

 	\begin{proof}
		Let $G=(V,E)$ be a graph with minimum degree $\delta\geq 2k^2$, for some fixed integer $k\geq 2$. Analogously as within the proof of Theorem \ref{bipartiteTheorem}, we shall use Lemma~\ref{mod_alon} in order to choose $k$ consecutive subgraphs $G_1,\dotsc,G_k$ of $G$ and colour the edges of each $G_i$ with colour $i$. The remaining edges shall be coloured with colour $k+1$. 
  
  Set $\overline{G}_0 = G$. For $i = 1,\dotsc,k$ let $G_i$ be a subgraph of $\overline{G}_{i-1}$ induced by the edges $e$ with $x_i(e)=1$ where $x_i$ is a function resulting from applying Lemma~\ref{mod_alon} to $\overline{G}_{i-1}$ with a constant weight function $z_i:E(\overline{G}_{i-1})\rightarrow \{\alpha_i\}$, where the value of $\alpha_i \in [0,1]$ shall be specified later; we also set $\overline{G}_i = (V(G), E(\overline{G}_{i-1}) \setminus E(G_i))$. Let $G_{k+1} = \overline{G}_k$. 
  We shall show we may choose $\alpha_i$ for $i \in \{ 1, \dotsc, k\}$ so that for every vertex $v$ of $G$ and each $j \in \{1,\dotsc,k+1\}$,
		\begin{equation}\label{dGjineqdG}
  d_{G_j}(v) \leq \frac{d_G(v)}{k}.
  \end{equation}
		
  By Lemma~\ref{mod_alon} (i), for every $v\in V$: 
		\begin{equation}\label{dG1inequalities}
  \alpha_1 d_G(v) -1 < \sum\limits_{e\ni v} x_1(e) = d_{G_1}(v) \leq \alpha_1 d_G(v) +1 \text{.}
  \end{equation}
  In order to satisfy~(\ref{dGjineqdG}) it is thus necessary and sufficient to
  choose $\alpha_1$ so that $\alpha_1 d_G(v) +1 \leq \frac{d_G(v)}{k}$ for all vertices $v$ of $G$, that is $\alpha_1 \leq \frac{1}{k} - \frac{1}{d_G(v)}$. Since the function $f(n) = \frac{1}{k}-\frac{1}{n}$ is increasing for $n>0$, we shall achieve our goal by setting
		\begin{equation}\label{Alpha1setting}
  \alpha_1 = \frac{\frac{\delta}{k}-1}{\delta} \text{.} 
  \end{equation}
		Consequently, a vertex $v$ of degree $\delta$, which in some sense are the most restrictive ones, 
  may in theory end up with $d_{G_1}(v)$ arbitrarily close to $\frac{\delta}{k}-2$, cf.~(\ref{dG1inequalities}) and~(\ref{Alpha1setting}). Hence for $i \in \{ 1, \dotsc, k\}$ we in general set:
	\begin{equation}\label{Alphaisetting}		
   \alpha_i = \frac{\frac{\delta}{k} -1}{\delta-(i-1)(\frac{\delta}{k}-2)} \text{.} 
  \end{equation}

  We shall now formally prove that such choices of $\alpha_i$ guarantee (\ref{dGjineqdG}) to hold for all $j$ and $v$.
		Let $v$ be an arbitrarily chosen vertex of $G$. There exists $\beta \geq 1$ such that $d_G(v) = \beta \delta$. 
  We shall precisely show that for every $i \in \{1,\dotsc,k\}$,
		\begin{equation}\label{bound1}
			d_{\overline{G}_i}(v) \leq \beta \left( \delta - i \left( \frac{\delta}{k} - 2 \right) \right)
		\end{equation}
		and
		\begin{equation}\label{bound2}
			d_{G_i}(v) \leq \frac{\beta \delta}{k} = \frac{d_G(v)}{k} \text{.}
		\end{equation}
  
		We proceed by induction with respect to $i$. Since
		\[ \alpha_1 d_G(v) = 
  \frac{\frac{\delta}{k}-1}{\delta} \cdot \beta \delta = \beta \left( \frac{\delta}{k}-1 \right) \]
		and  $\beta \geq 1$, 
  by~(\ref{dG1inequalities}) we obtain
		\[ \beta \left( \frac{\delta}{k}-2 \right) < d_{G_1}(v) \leq \beta \frac{\delta}{k} \text{,}  \]
		so \eqref{bound1} and \eqref{bound2} hold for $i=1$, which yields the base case of induction.
		
		For an induction step, 
  assume that 
  $d_{\overline{G}_{i-1}} (v) = D \leq \beta \left( \delta - (i-1) \left( \frac{\delta}{k} - 2 \right) \right)$ for some $i\leq k$. Note that by Lemma~\ref{mod_alon} (i), 
		\[ d_{G_i}(v) > \alpha_i D -1 \text{,} \]
		and thus, by~(\ref{Alphaisetting}):
		\begin{align*}
			d_{\overline{G}_i}(v) &< D - (\alpha_i D -1) = (1 - \alpha_i) D + 1 \\
			&\leq (1 - \alpha_i) \beta \left( \delta - (i-1) \left( \frac{\delta}{k} - 2 \right) \right) + 1 \\
			&= \beta \left( \delta - (i-1) \left( \frac{\delta}{k} - 2 \right) - \left(\frac{\delta}{k} -1 \right) \right) +1\\
			&\leq  \beta \left(\delta - (i-1) \left( \frac{\delta}{k} - 2 \right) - \left( \frac{\delta}{k} - 1 \right) \right) + \beta \\ 
			&= \beta \left(\delta - i \left( \frac{\delta}{k} - 2 \right) \right) \text{,}
		\end{align*}
cf.~\eqref{bound1}.		
		
		On the other hand, by~(\ref{dG1inequalities}) and~(\ref{Alphaisetting}):
		\begin{align*}
			d_{G_i}(v) &\leq \alpha_i D + 1 \leq \alpha_i \beta \left( \delta - (i-1) \left( \frac{\delta}{k} - 2 \right) \right) +1 \\
			&= \beta \left( \frac{\delta}{k} - 1 \right) +1 \leq \beta \left( \frac{\delta}{k} - 1 \right) + \beta =  \frac{\beta \delta}{k} = \frac{d_G(v)}{k} \text{,}
		\end{align*}
		hence \eqref{bound2} and \eqref{bound1} hold
  (where \eqref{bound2} implies~(\ref{dGjineqdG})). 
		
		Finally, observe that by \eqref{bound1} we have:
		\[ d_{G_{k+1}}(v) = d_{\overline{G}_{k}}(v) \leq \beta \left( \delta - k \left( \frac{\delta}{k} - 2 \right) \right) = 2 \beta k \text{.} \]
		Since $\delta \geq 2k^2$, we obtain $d_{G_{k+1}}(v) \leq \frac{\beta\delta}{k} = \frac{d_G(v)}{k}$, which concludes the proof of Theorem~\ref{thm1-GenMajEdg}.
	\end{proof}
	
	Note that the bound on the minimum degree of $G$ was only required to bound $d_{G_{k+1}}$ in the proof of Theorem~\ref{thm1-GenMajEdg}. This remark allows us to use almost entire reasoning above within the proof of Theorem~\ref{thm2-GenMajEdg} below, which improves the general lower bound for $\delta_k^{\rm opt}$. 
 This refinement exploits the following straightforward and direct consequence of Euler's Theorem (through adding a single auxiliary vertex to a graph, if necessary). Details of its proof can be found e.g. in~\cite{majority23,Przybylo22standard} and most likely in many other papers.

	\begin{observation} 
 \label{euler}
 Let $G$ be a connected graph.
		\begin{enumerate}
			\item[$(1^\circ)$] If $G$ has an even number of edges or $G$ contains vertices of odd degree, then 
   $G$ has a $2$-edge-colouring such that 
   for every vertex $u$ of $G$, at most $\left\lceil \frac{d_G(u)}{2} \right\rceil$ of the edges incident with $u$ have the same colour.
			\item[$(2^\circ)$] If $G$ has an odd number of edges, all vertices of $G$ have even degree and $u_G$ is any vertex of $G$, then $G$ has a $2$-edge-colouring such that 
   for every vertex $u$ of $G$ distinct from $u_G$, exactly $\frac{d_G(u)}{2}$ of the edges incident with $u$ have the same colour, and at most $\frac{d_G(u_G)}{2}+1$ of the edges incident with $u_G$ have the same colour.
		\end{enumerate}
	\end{observation}
	
	In what follows, 
 a \textit{bad vertex} shall mean a vertex of $G$ which was chosen as the vertex $u_G$ while applying Observation~\ref{euler} above, that is the vertex with exactly $\frac{d_G(v)}{2}+1$ incident edges coloured the same in one of the two colours. 
 
	\begin{theorem} \label{thm2-GenMajEdg}
 Let $k = 2^n + m - 1\geq 2$ where $n$ is a positive integer and $m$ is a nonnegative integer less than $2^n$. If $G$ is a graph with minimum degree $\delta\geq \frac{3}{2}k^2+\frac{1}{2}km+\frac{1}{2}k$, 
 then $G$ has a $\frac{1}{k}$-majority $(k+1)$-edge-colouring.
	\end{theorem} 
	\begin{proof}
		We start by partially colouring the graph $G=(V,E)$ with $m$
  colours, 
  choosing $G_1,\ldots,G_m$, corresponding to colours $1,\ldots,m$, the same way as in the proof of Theorem \ref{thm1-GenMajEdg}. 
  By~\eqref{bound2} these colours satisfy the \textit{majority rule}, that is for every vertex $u$, at most $d_G(u)/k$ 
  edges incident with $u$ are coloured with any of the colours in $\{1,\ldots,m\}$. 
  
  Let $H$ be a subgraph of $G$ induced by the uncoloured edges. Notice that $H=\overline{G}_m$ (using the notation from the proof of Theorem \ref{thm1-GenMajEdg}), and thus by~\eqref{bound1}, if $v$ is a vertex of $G$ such that $d_G(v)=\beta \delta$, then
		\begin{equation}\label{dHvInequ}
		    d_H(v) \leq \beta \left( \delta - m \left( \frac{\delta}{k} - 2 \right) \right) \text{.} 
		\end{equation} 
We shall colour the edges of $H$ with new $2^n$ colours, namely the elements of the set $\{0,1\}^n$, hence we shall be colouring these edges with binary vectors of length $n$. For any vector $w\in\{0,1\}^n$ and $0\leq j\leq n$ we denote by $[w]_j$ the \emph{prefix} of length $j$ of $w$, that is the vector in $\{0,1\}^j$ formed of $j$ first consecutive coordinates of $w$ (where $[w]_j=\emptyset$ for $j=0$).
Initially (in step $0$) we associate the vector $c_e=(0,\ldots,0)$ to every edge $e$ of $H$. The vectors $c_e$, $e\in E(H)$, shall be modified one coordinate after another in $n$ steps. In certain situations we shall however be finally fixing all the remaining coordinates of some of these vectors at once -- the corresponding edges shall be called \emph{determined}. In what follows, $c_e$ shall always refer to the current value of the colour (vector) associated with an edge $e$.
Suppose for a given $i\in\{1,\ldots,n\}$ we have completed step $i-1$ of our construction, hence each $c_e$ has the first $i-1$ coordinates finally fixed (or all, for selected, determined, $e\in E(H)$), and we are about to perform step $i$. We proceed as follows.

For each possible prefix $p\in\{0,1\}^{i-1}$ we denote by $H_p$ the subgraph induced in $H$ by all not yet determined edges $e$ with $[c_e]_{i-1}=p$ (after step $i-1$). In each such $H_p$ we consider all components one after another. Let $H'$ be such a component for any fixed $p\in\{0,1\}^{i-1}$. Let us give an advance notice to the fact that at most one vertex of $H'$ shall be chosen to be so-called \emph{special for $H_{p'}$}, where $p'$ is the extensions of $p$ with $1$ added to its end (i.e. $p'\in\{0,1\}^i$, $[p']_{i-1}=p$ and $p'(i)=1$), according to the rule specified below.
\begin{enumerate}
    \item[(a)] If for each vertex $v$ of $H'$ there exists a prefix $q$ of $p$ (possibly $q=p$) such that $v$ is special for $H_q$, 
    then for every edge $e$ of $H'$ we fix as $0$ all the remaining (starting from the $i$'th one) coordinates of $c_e$,  
    hence all edges of $H'$ become determined.
    \item[(b)] Otherwise, we use Observation~\ref{euler} to temporarily colour the edges of $H'$ blue and red. For each edge $e$ of $H'$ we fix the $i$'th coordinate of $c_e$ as $0$ if $e$ is blue, and $1$ otherwise. Moreover, if we are forced to create a bad vertex $u_{H'}$ (with $(d_{H'}(u_{H'})/2)+1$ incident edges of the same colour), then we choose it so that it was not special for $H_q$ for any prefix $q$ of $p$ and assign colours blue and red so that $u_{H'}$ is incident with exactly $(d_{H'}(u_{H'})/2)+1$ red edges; we then also choose $u_{H'}$ to be special for $H_{p'}$ where $p'$ is the extensions of $p$ with $1$ added to its end.
\end{enumerate}

	After going through all $n$ steps described above, we complete a $(k+1)$-edge-colouring of the graph $G$. It remains to show that all new colours satisfy the $\frac{1}{k}$-majority rule.
 Consider a vertex $v\in V$ and any fixed colour $\alpha\in\{0,1\}^n$. Let us denote by $G_\alpha$ the subgraph induced in $G$ (in fact in $H$) by all the edges coloured with $\alpha$.

 Suppose first that $v$ is incident with some edge $e\in E(H)$ which was coloured (determined) with a colour $\alpha$ according to Rule (a) above, i.e. at certain iteration $i$, the edge $e$ belonged to a component $H'$ of some $H_p$ with all vertices being special for some $H_q$ where $q$ is a prefix of $H_p$. Note however that by our construction, $H'$ must have been a (connected) subgraph of every such $H_q$, and thus for each such $H_q$ at most one vertex of $H'$ might have been chosen to be special for $H_q$. Consequently, as $p$ has no more than $n$ distinct prefixes,  $H'$ must have contained at most $n$ vertices. Hence for each its vertex, in particular $v$,
 $$d_{G_\alpha}(v) < n \leq k \leq\frac{d_G(v)}{k}.$$

Assume in turn that every edge $e\in E(G_\alpha)$ incident with $v$ was coloured by means of Rule (b) exclusively. 
This rule was thus utilised $n$ times in order to settle all edges incident with $v$ coloured $\alpha$, each time via application of Observation~\ref{euler} to a component $H'$ (containing $v$) of some $H_p$, where $p$ is a prefix of $\alpha$. 
Suppose $v$ had degree $d$ in such $H'$, say in iteration $i$. 
If $v$ was chosen to be special for $H_p$, which could happen only once during $n$ steps of our construction (for prefixes $p$ of the fixed $\alpha$), then at most $\frac{d}{2}+1$ edges incident with $v$ got their colours' prefixes fixed as $p$ after step $i$; otherwise the number of such edges is bounded above by 
$\frac{d}{2} + \frac{1}{2}$, cf. Observation~\ref{euler}. Only such edges retained the chance to belong to $G_\alpha$.
In order to estimate the maximum number of edges incident with $v$ which eventually could be coloured $\alpha$ let us thus consider two following functions.
 Let $f(d) = \frac{d}{2} + \frac{1}{2}$ and $g(d) = \frac{d}{2} + 1$. By the observations above, $d_{G_\alpha}(v)$ 
 is bounded above
 by the maximum of $f^n(d)$ and $f^i(g(f^j(d)))$ for all natural numbers $i$ and $j$ such that $i+j = n-1$ where $d=d_H(v)$. Since $f(d)<g(d)$ for all $d$, the value of $f^i(g(f^j(d)))$ is greater that $f^n(d)$ for all $i$ and $j$ satisfying $i+j=n-1$. We shall prove the following upper bound.
	
	\begin{claim} \label{obs6}
		 The inequality $f^i(g(f^j(d))) \leq \frac{d-1}{2^n} + \frac{3}{2}$ holds for all $d$ and all natural numbers $i$ and $j$ such that $i+j=n-1$.
   \end{claim}
		 \begin{proof}[Proof of Claim~\ref{obs6}] 
		 	We begin by proving that
		 	\[ f^i(d) = \frac{d-1}{2^i} + 1 \]
		 	holds for any nonnegative integer $i$.
		 	We proceed  by induction with respect to $i$. Clearly $f^0(d) = d = \frac{d-1}{2^0} + 1$. Assume that $f^j(d) = \frac{d-1}{2^j} + 1$ holds for all $j<i$. Thus,
		 	\[ f^i(d) = f(f^{i-1}(d)) = f\left( \frac{d-1}{2^{i-1}} + 1 \right) =  \frac{\frac{d-1}{2^{i-1}} + 1}{2} + \frac{1}{2} = \frac{d-1}{2^{i}} + 1 \text{.} \]
		 	
		 	Finally, for any fixed $i$ and $j$ such that $i+j=n-1$, we have
		 	\begin{align*}
		 		f^i(g(f^j(d))) &= f^i \left( g \left( \frac{d-1}{2^j} + 1 \right)\right) = f^i \left( \frac{\frac{d-1}{2^j} + 1}{2} + 1 \right) \\
		 		&= f^i \left( \frac{d-1}{2^{j+1}} + \frac{3}{2} \right) = \frac{\frac{d-1}{2^{j+1}} + \frac{3}{2}-1}{2^i} + 1 = \frac{d-1}{2^n} + \frac{1}{2^{i+1}} + 1 \text{.}
		 	\end{align*}
		 	The value of $f^i(g(f^j(d)))$ is greatest when $i=0$, and thus $f^i(g(f^j(d))) \leq \frac{d-1}{2^n} + \frac{3}{2}$.
		 \end{proof}

 By Claim~\ref{obs6} and discussion above we obtain that 
 \begin{equation}\label{dGAlphaIneq}
 d_{G_\alpha}(v)\leq \frac{d_H(v)-1}{2^n} + \frac{3}{2}.
 \end{equation}
	It remains to show that if $d_G(v)=\beta\delta$, then $d_{G_\alpha}(v)\leq \frac{\beta\delta}{k}$.
Recall that by~\eqref{dHvInequ}, $d_H(v) \leq \beta \left( \delta - m \left( \frac{\delta}{k} - 2 \right) \right)$. This combined with~\eqref{dGAlphaIneq} yield the following, where we make use of 
the facts that $k = 2^n + m - 1$, $\delta \geq \frac{3}{2}k^2+\frac{1}{2}km+\frac{1}{2}k$ and
$\beta \geq 1$:
\begin{align*}  
  d_{G_\alpha}(v) &\leq \frac{\beta \left( \delta - m \left( \frac{\delta}{k} - 2 \right) \right) - 1}{2^n} + \frac{3}{2} \\
		&= \frac{\beta \delta \left( 1-\frac{m}{k} \right) + 2\beta m - 1}{2^n} + \frac{3}{2} \\ 
		&= \frac{\beta \delta \left( \frac{2^n - 1}{k} \right) + 2\beta m - 1}{2^n} + \frac{3}{2} 
  \\
  &= \frac{\beta \delta}{k} + \frac{2\beta m - 1 - \frac{\beta \delta}{k}}{2^n} + \frac{3}{2} \\ 
		&\leq \frac{\beta \delta}{k} + \frac{2\beta m - 1 - \beta \left( \frac{3}{2}k + \frac{1}{2}m + \frac{1}{2} \right)}{2^n} + \frac{3}{2} \\ 
		&= \frac{\beta \delta}{k} + \frac{\beta \left( - \frac{3}{2}k + \frac{3}{2}m - \frac{1}{2} \right) - 1}{2^n} + \frac{3}{2} \\
		&= \frac{\beta \delta}{k} + \frac{\beta \left( - \frac{3}{2} \left( 2^n - 1 \right) - \frac{1}{2} \right) - 1}{2^n} + \frac{3}{2} \\
		&= \frac{\beta \delta}{k} - \frac{3}{2}\beta + \frac{\beta - 1}{2^n} + \frac{3}{2} \\
		&= \frac{\beta \delta}{k} + (1-\beta) \left( \frac{3}{2} - \frac{1}{2^n} \right) \\
  &\leq \frac{\beta \delta}{k} = \frac{d_G(v)}{k}. 
	\end{align*}
This concludes the proof of Theorem~\ref{thm2-GenMajEdg}.
\end{proof}

	Note that the formula for $k$ used in Theorem~\ref{thm2-GenMajEdg} implies that $m$ is always bounded above by $\frac{k}{2}$, and thus we immediately obtain the following corollary.
	
	\begin{corollary}\label{7_4_Corollary}
		For every integer $k\geq 2$, if a graph $G$ has minimum degree $\delta\geq \frac{7}{4}k^2+\frac{1}{2}k$, then $G$ has a $\frac{1}{k}$-majority $(k+1)$-edge-colouring.
	\end{corollary}

\section{Confirmation of Conjecture~\ref{EdgeMajconj} for initial values of $k$}\label{SectionSmallCases}

	The main result of \cite{majority23} confirms Conjecture~\ref{EdgeMajconj} for $k=2$. In this section we shall 
 extend this result towards two following values of $k$. 
 To achieve this we shall use the two observations below.
	
	\begin{observation} \label{limit_deg_1} 
		Let $k\geq 2$ be an integer. If every graph with minimum degree $\delta \geq k^2$ and maximum degree $\Delta < 2k^2$ has a $\frac{1}{k}$-majority $(k+1)$-edge-colouring, then every graph with minimum degree at least $k^2$ has a $\frac{1}{k}$-majority $(k+1)$-edge-colouring.
		\begin{proof}
			Let $G$ be an arbitrary graph with minimum degree at least $k^2$.  
   If the maximum degree of $G$ is less than $2k^2$ then by assumption it has a $\frac{1}{k}$-majority $(k+1)$-edge-colouring. Otherwise, let $v$ be a vertex of $G$ such that $d_G(v) \geq 2k^2$. There exist unique integers $n$ and $d$ such that $d_G(v) = n k^2 + d$ and $k^2 \leq d < 2k^2$. Partition the neighbourhood of $v$ into $n+1$ disjoint sets $N_0,\dotsc,N_n$ such that $|N_0| = d$ and $|N_i|=k^2$ for $i\geq 1$. Let $H$ be a graph such that $V(H) = V(G)\setminus\{v\} \cup \{v_0,v_1,\dotsc,v_n\}$ and $E(H) = E(G-v)\cup \bigcup\limits_{i=0}^{n} \{uv_i : u\in N_i\}$. Note that this operation yields a natural bijection between the edges of $G$ and the edges of $H$. Let $\overline{G}$ be a graph constructed from $G$ by applying the above operation to all vertices of $G$ with degree at least $2k^2$. By construction, the maximum degree of $\overline{G}$ is less than $2k^2$, hence $\overline{G}$ has a $\frac{1}{k}$-majority $(k+1)$-edge-colouring, which yields a $(k+1)$-edge-colouring of $G$. It remains to prove that this is also a $\frac{1}{k}$-majority edge-colouring of $G$.
			
			Let $v$ be a vertex of $G$ with degree $d_G(v) = n k^2 + d$ where $k^2 \leq d < 2k^2$ (with $n$ possibly equal $0$). The number of edges adjacent to $v$ coloured with the same colour is bounded above by
			\[ n \left\lfloor \frac{k^2}{k} \right\rfloor + \left\lfloor \frac{d}{k} \right\rfloor = nk + \left\lfloor \frac{d}{k} \right\rfloor = \left\lfloor \frac{nk^2 + d}{k} \right\rfloor = \left\lfloor \frac{d_G(v)}{k} \right\rfloor \text{,} \]
			hence the colouring of $G$ is indeed a $\frac{1}{k}$-majority $(k+1)$-edge-colouring.
		\end{proof}
	\end{observation}

	\begin{observation} \label{limit_deg_2}
		For every integer $k\geq 2$, let $S_k$ be the set of all integers $i$ between $k^2$ and $2k^2$ such that $i \equiv k-1 \pmod k$. Let $\mathcal{G}_k$ be the set of all graphs for which the degrees of all vertices are in the set $S_k$. If every graph in $\mathcal{G}_k$ has a $\frac{1}{k}$-majority $(k+1)$-edge-colouring, then every graph with minimum degree
  at least
  $k^2$ has a $\frac{1}{k}$-majority $(k+1)$-edge-colouring.
		\begin{proof}
			By Observation~\ref{limit_deg_1} it is sufficient to 
   consider 
   graphs with maximum degree less than $2k^2$. Let $G$ be an arbitrary graph with minimum degree $\delta \geq k^2$ and maximum degree $\Delta<2k^2$. If all vertices of $G$ have degrees in the set $S_k$, then by assumption $G$ has a $\frac{1}{k}$-majority $(k+1)$-edge-colouring. Otherwise, let $H$ be a graph constructed from $G$ by taking two copies of $G$ and joining by edges the vertices of $G$ which do not have degrees in the set $S_k$ with their corresponding counterparts in the second copy of $G$. Note that $G$ is a subgraph of $H$. Moreover, for every vertex $v$ of $G$, either $d_H(v)=d_G(v) \in S_k$ or $d_G(v) \equiv i \pmod k$ and $d_H(v) \equiv i+1 \pmod k$ (and the same holds for the vertices in the second copy of $G$). We repeat this operation until all vertices of the obtained graph $\overline{G}$
   have degrees in the set $S_k$.
   Note the degree of each vertex $v$ of $G$ increased by at most $k-1$, and more importantly,
   \begin{equation}\label{dHvdGvFloorsTheSame}
   \left\lfloor \frac{d_H(v)}{k} \right\rfloor = \left\lfloor \frac{d_G(v)}{k} \right\rfloor.
   \end{equation}
   Since the maximum degree of $G$ is at most $2k^2-1 \equiv k-1 \pmod k$, the maximum degree of $\overline{G}$ is also less than $2k^2$, hence $\overline{G}\in \mathcal{G}_k$. By our assumption, there is a $\frac{1}{k}$-majority $(k+1)$-edge-colouring $c$ of $\overline{G}$. 
   Since $G$ is a subgraph of $\overline{G}$, by~\eqref{dHvdGvFloorsTheSame}, the 
   colouring $c$ restricted to the edges of $G$
   yields a $\frac{1}{k}$-majority $(k+1)$-edge-colouring of $G$.
		\end{proof}
	\end{observation}
	
	Observations~\ref{limit_deg_1} and~\ref{limit_deg_2} allow us to narrow down the set of graphs we need to consider in order to prove 
 Conjecture~\ref{EdgeMajconj}. 
 To start with, we exemplify their usefulness by reproving Theorem~\ref{majority23-res2}, whose proof provided in~\cite{majority23} is rather lengthy. Tools and observations introduced above yield a short and straightforward argument.

		\begin{proof}[Proof of Theorem~\ref{majority23-res2}]
			Let $G$ be an arbitrary graph with minimum degree $\delta\geq 4$. By Observation~\ref{limit_deg_2},
we can assume that the degrees of all vertices of $G$ are in the set $S_2 = \{5,7\}$. Let $D_2$ be the set of vertices of degree $5$, and $D_3$ the set of vertices of degree $7$. Vertices in $D_2$ can have at most $2$ incident edges in the same colour, and vertices of $D_3$ can have at most $3$ such edges. We shall construct a majority $3$-edge-colouring of $G$ in two steps. First, use Lemma~\ref{mod_alon} with a weight function assigning $1/3$ to every edge of $G$ to colour some edges of $G$ with one of the three colours (similarly as in the proof of Theorem \ref{thm1-GenMajEdg}). Vertices in $D_2$ have $1$ or $2$ edges coloured, and vertices in $D_3$ -- $2$ or $3$. Let $H$ be the graph induced by uncoloured edges of $G$. Vertices in $D_2$ have degrees $3$ or $4$ in $H$, and vertices in $D_3$ have degrees $4$ or $5$. Finally, use Observation~\ref{euler} to colour the edges of $H$ with the remaining two colours. Note that every component of $H$ either has vertices of odd degree or all of its vertices have degree $4$ and thus such component has an even number of edges. 
			Hence, by Observation~\ref{euler}, at most $\left\lceil \frac{d_H(u)}{2} \right\rceil$ of the edges incident with any given vertex $u$ shall get the same colour, which satisfies the majority rule for the graph $G$. 
		\end{proof}

	\begin{theorem} \label{small3}
		Every graph with minimum degree $\delta\geq 9$ has a $\frac{1}{3}$-majority 4-edge-colouring.
		\begin{proof}
			Let $G$ be a graph with minimum degree $\delta\geq 9$. By Observation~\ref{limit_deg_2}, we can assume that the degrees of all vertices of $G$ are in the set $S_3 = \{11,14,17\}$. Similarly as before, let $D_3$ be the set of vertices of degree $11$ (which can have at most $3$ incident edges with the same colour), $D_4$ be the set of vertices of degree $14$ (allowing $4$ incident monochromatic edges), and $D_5$ -- vertices of degree $17$ (allowing $5$ incident edges with the same colour). Let $G'$ be a graph constructed from the graph $G$ by removing all components which have all vertices of degree $14$ and an odd number of edges. Hence, each component of $G'$ has an even number of edges or contains vertices of odd degree. Colour the edges of $G'$ using Observation~\ref{euler} with colours blue and red. The number of incident edges with the same colour 
   shall equal $5$ or $6$ for vertices in $D_3$, $7$ for vertices in $D_4$, and $8$ or $9$ for vertices in $D_5$.
			
			We shall show that 
   in fact we can choose such a $2$-edge-colouring of $G'$ complying with Observation~\ref{euler} that neither of the two colours induces a component  
   with an odd number of edges and all vertices of degree $6$. Assume this is not the case and consider a $2$-edge-colouring of $G'$ consistent with Observation~\ref{euler} with the least number of such bad components. Without loss of generality we can assume that there exists such a bad component, say $H$ in the graph induced by the blue edges. Notice that $H$ is in particular Eulerian, and thus it is $2$-edge-connected. Clearly, all the vertices in $H$ are in the set $D_3$ and have exactly 5 red incident edges (in $G'$). Let $v$ be an arbitrary vertex of $H$, and let $u_1, u_2$ be any two distinct neighbours of $v$ in $H$. Consider a component in the subgraph of $G'$ induced by the red edges such that $v$ is in this component, denote it $H'$. If $u_1$ is not in the same (red) component as $v$, then we can recolour the edge $u_1v$ with red colour. In such a case, $H$ shall no longer have exclusively vertices of degree $6$, and no new $6$-regular monochromatic component shall be created, since at least one other than $u_1$ vertex in $H'$ needs to have odd degree. 
			We proceed similarly if $u_1$ is in the same red component as $v$, but $u_2$ is not. If both $u_1$ and $u_2$ are in the same red component as $v$, then we can recolour the edge $u_1v$ to red colour. Then, neither $H$ nor $H'$ shall be a $6$-regular components. Hence, in each case, the number of monochromatic components with an odd number of edges and all vertices of degree $6$ can be decreased, which is in contradiction with the assumption that our colouring had the least possible number of bad components.
			
			As a result, both in the graph induced by the red edges and in the graph induced by the blue edges each component has an even number of edges or contains vertices of odd degree or contains a vertex of degree $8$. Hence, we can again use Observation~\ref{euler} (separately for graphs induced by both of the colours), choosing a vertex of degree $8$ as the bad vertex if necessary. The $4$-edge-colouring obtained this way satisfies the $\frac{1}{3}$-majority rule for the graph $G'$.
			
			Finally, consider components of $G$ with all vertices of degree $14$ and an odd number of edges. Using Observation~\ref{euler} we obtain a $2$-edge-colouring of such components with colours red and blue, such that in the subgraph generated by the edges of one of the colours all vertices shall have degree $7$, except one vertex of degree $6$ or $8$. In either case, there shall be a vertex of odd degree in each of the obtained monochromatic components, hence using again Observation~\ref{euler} (and merging the result with the colouring of $G'$) yields a $\frac{1}{3}$-majority $4$-edge-colouring of $G$.
		\end{proof}
	\end{theorem}
	
	\begin{theorem}
		Every graph with minimum degree $\delta\geq 16$ has a $\frac{1}{4}$-majority $5$-edge-colouring.
		\begin{proof}
			Let $G$ be a graph with minimum degree $\delta\geq 16$. By Observation~\ref{limit_deg_2}, we can assume that the degrees of all vertices of $G$ are in the set $S_4 = \{19,23,27,31\}$. Let $D_4$ be the set of vertices of degree $19$ (which can have at most $4$ incident edges with the same colour), $D_5$ be the set of vertices of degree $23$ (allowing $5$ incident monochromatic edges), $D_6$ be the set of vertices of degree $27$ (allowing $6$ incident edges with the same colour), and $D_7$ -- vertices of degree $31$ (allowing $7$ incident monochromatic edges). First, use Lemma~\ref{mod_alon} with a weight function assigning $1/5$ to all edges of $G$ to colour some edges of $G$ with colour 1. As a result, every vertex $v$ of $G$ has either $\lfloor \frac{d(v)}{5} \rfloor$ or $\lceil \frac{d(v)}{5} \rceil$ incident edges coloured $1$. As none of the vertices has degree divisible by $5$, these two values are distinct, and by Lemma~\ref{mod_alon} (ii), every edge $uv$ with exactly $\lfloor \frac{d(u)}{5} \rfloor$ edges incident with $u$ coloured $1$ and exactly $\lfloor \frac{d(v)}{5} \rfloor$ edges incident with $v$ coloured $1$ must be coloured $1$ as well.

			Let $H$ be the subgraph of $G$ induced by the uncoloured edges. Vertices in $D_4$ have degrees in $\{15,16\}$ in $H$, vertices in $D_5$ -- degrees in $\{18,19\}$, vertices in $D_6$ -- degrees in $\{21,22\}$, and vertices in $D_7$ -- degrees in $\{24,25\}$. Using Observation~\ref{euler} divide the graph $H$ into two subgraphs $H_1$ (coloured blue) and $H_2$ (coloured red), choosing vertices of degree $18$ or $22$ as the bad vertices, if necessary. In the components of the graphs $H_1$ and $H_2$ we have the following situation: vertices in $D_4$ have degrees in $\{7,8\}$, vertices in $D_5$ have degrees in $\{9,10\}$ (and possibly a single vertex has degree $8$), vertices in $D_6$ have degrees in $\{10,11\}$ (and possibly a single vertex has degree $12$), and vertices in $D_7$ have degrees in $\{12,13\}$. 
			Notice that if a vertex in $D_6$ has degree $12$, then in the graph $H$, it had to be in a component with no vertex of odd degree, hence in the graphs $H_1$ and $H_2$, this vertex cannot be in the same component as any vertex of degree $10$.
			
			We shall show that we can recolour the graph $H$ (retaining conditions mentioned above) in such a way that neither $H_1$ nor $H_2$ contains a component with an odd number of edges whose every vertex either has degree $10$ and belongs to $D_5$ or has degree $8$ and belongs to $D_4$.
   Assume this is not possible and consider a colouring with the least number of such components. Let $\overline{H}$ be one of these components. Since the number of edges of $\overline{H}$ is odd, at least one of the vertices in $\overline{H}$ must have degree $10$. Let $v$ be a vertex of degree $10$ in $\overline{H}$ and let $u_1$ and $u_2$ be two distinct neighbours of $v$ in $\overline{H}$. 
   If $v$ is the only vertex of degree $10$ in $\overline{H}$ and $d_H(v)=18$, then $v$ is a bad vertex and for all the remaining vertices $u$ of $\overline{H}$ we must have $d_H(u)=16=\lceil \frac{d(u)}{5} \rceil$, and two such vertices must be adjacent in 
   $\overline{H}$ (hence also in $H$), which is impossible by the last remark of the first paragraph of the proof.
   We may thus assume that 
   $d_H(v)=19=\lceil \frac{d(v)}{5} \rceil$, and hence, again by the last remark of the first paragraph of the proof,  
   $u_1$ and $u_2$ must be vertices of degree $8$ in $\overline{H}$ and $15$ in $H$. 
   Thus, proceeding analogously as in the proof of Theorem~\ref{small3} we can obtain a colouring of $H$ with a smaller number of bad components, a contradiction. 
			
			Hence, each component of $H_1$ and $H_2$ contains vertices of odd degree or has an even number of edges or contains a vertex of degree $10$ belonging to  $D_6$ or a vertex of degree $12$ belonging to $D_7$. Thus, using Observation~\ref{euler} (with one of the mentioned vertices being chosen as the bad vertex, if necessary) we can obtain a $4$-edge-colouring of $H$, which completes a $\frac{1}{4}$-majority $5$-edge-colouring of $G$. 
		\end{proof}
	\end{theorem}

\section{Concluding remarks}

Theorem~\ref{thm1-GenMajEdg} and the construction in Observation~\ref{ExampleGeneral} imply we managed to settle the order of magnitude of our main objective: $\delta_k^{\rm opt}$ and approximate it  
within a multiplicative factor of $2$.
Our Conjecture~\ref{EdgeMajconj} clearly conveys  
we expect  
that $2$ is a redundant factor in our $2k^2$ upper bound for $\delta_k^{\rm opt}$. In fact, Corollary~\ref{7_4_Corollary} shows that the leading factor in this bound should not be larger than $\frac{7}{4}k^2$. 
Moreover, Theorem~\ref{thm2-GenMajEdg} also implies that there is e.g. an infinite sequence of values of $k$ for which $\delta_k^{\rm opt}\leq (\frac{3}{2}+o(1))k^2$.

On the other hand, even though we were able to confirm the conjecture for several initial values of $k$ in Section~\ref{SectionSmallCases}, 
we are not entirely 
convinced that the postulated quantity of $\delta_k^{\rm opt}$ has to be 
precisely correct for all $k$. 
One may possibly come up 
with some more sophisticated construction than the one in Observation~\ref{ExampleGeneral}, and this seems an interesting 
direction to be more thoroughly investigated.
However, we would not expect the lower bound stemming from such a potential construction to exceed $k^2$ by far. 
In any case we strongly expect an upper bound of the form $(1+o(1))k^2$ to be valid for $\delta_k^{\rm opt}$.

Recall that Observation~\ref{ExampleGeneral} implies we cannot directly extend to all graphs our Theorem~\ref{bipartiteTheorem},  yielding an optimal solution for the family of bipartite graphs.
However, as mentioned, the main obstacle on the way towards obtaining some form of such an extension was 
lack of a
valid in the general setting
correspondent of Claim~\ref{Claim1-bipartite} from the proof of Theorem~\ref{bipartiteTheorem}, where 
in some sense it allowed us to control and `capture' degrees of consecutively constructed subgraphs of a given bipartite graph 
within a reasonably narrow interval.
In fact, in pursuit of such a correspondent
we came up with our refinement of the lemma of Alon and Wei~\cite{alon}, that is Lemma~\ref{mod_alon}.
Even though some aspects of this slight improvement were useful and crucial in the case of bipartite graphs, we did not use 
it in full measure, while at the same time it was not strong enough to provide a result we expect in a general case.
Nevertheless, we decided to include in our paper this
slightly excessive form of Lemma~\ref{mod_alon},
as a suggestion for possible further development of this tool, which might hopefully lead to solving Conjecture~\ref{EdgeMajconj}, or at least help closing 
the current gap.

Let us also
mention we believe that even solving our Problem~\ref{GeneralizedMajorityProblem} for the first open case of $k=5$ (and maybe some consecutive initial ones) 
seems interesting by itself,
and may furthermore  
shed light on a possible approach to attack Conjecture~\ref{EdgeMajconj} in its entirety.

Finally, let us remark why we believe the probabilistic approach seems difficult to be (directly) utilised while trying to prove Conjecture~\ref{EdgeMajconj}. It stems from the fact that if a graph has a vertex $v$ with degree (close to) $k^2$, then while colouring its edges randomly with $k+1$ colours we expect every colour to appear roughly $d(v)/(k+1)>k-1$ times (and in fact some colours must appear at least this many times) around $v$, while we admit at most $\lfloor d(v)/k\rfloor \leq k$ appearances of each colour. Thus in a way we admit an error of at most $1$ in frequency of appearing of each colour, which does not seem achievable via probabilistic approach, as e.g. typical concentration tools require admitting an error ``slightly'' larger than $\sqrt{(d(v)/k)}$ (which is enough as long as $d(v)$ is of magnitude roughly $k^3\log k$). This is also why we reckon that our, rather naive in nature, approach is surprisingly efficient.


\begin{thebibliography}{9}
	
	\bibitem{alon} Noga Alon and Fan Wei. Irregular subgraphs. \textit{Combin. Probab. Comput.} \textbf{32} (2023), no. 2, 269-283. doi:https://doi.org/10.1017/S0963548322000220 

\bibitem{Szabo-majority}
M. Anastos, A. Lamaison, R. Steiner, T. Szab\'o, 
Majority colorings of sparse digraphs, 
Electron. J. Combin. 28(2) (2021) \# P2.31.
doi:10.37236/10067

	\bibitem{digraph2} Marcin Anholcer, Bartłomiej Bosek and Jarosław Grytczuk. Majority Choosability of digraphs. \textit{Electron. J. Combin.} \textbf{24} (2017), no. 3, Paper No. 3.57, 5 pp. doi:https://doi.org/10.37236/6923
 
\bibitem{MajorityGeneralOur}
M. Anholcer, B. Bosek, J. Grytczuk, G. Gutowski, J. Przyby{\l}o, M. Zaj\k{a}c, Mrs. Correct and Majority Colorings, arXiv:2207.09739.


	\bibitem{majority23} Felix Bock, Rafał Kalinowski, Johannes Pardey, Monika Pilśniak, Dieter Rautenbach and Mariusz Woźniak. Majority Edge-Colorings of Graphs. \textit{Electron. J. Combin.} \textbf{30} (2023), no. 1, Paper No. 1.42, 8 pp. doi:https://doi.org/10.37236/11291
	
	\bibitem{digraph3} António Girão, Teeradej Kittipassorn and Kamil Popielarz. Generalized Majority Colourings of Digraphs. \textit{Combin. Probab. Comput.} \textbf{26} (2017), no. 6, 850-855. doi:https://doi.org/10.1017/S096354831700044X
	
\bibitem{Knox-Samal}
F. Knox, R. \v{S}\'amal, Linear bound for majority colourings of digraphs, Electron. J. Combin. 25(3) (2018) P3.29. doi:10.37236/6762

	\bibitem{digraph1} Stephan Kreutzer, Sang-il Oum, Paul Seymour, Dominic van der Zypen and David R. Wood. Majority Colourings of Digraphs. \textit{Electron. J. Combin.} \textbf{24} (2017), no. 2, Paper No. 2.25, 9pp.
	doi: 10.37236/6410
 
\bibitem{Lovasz-majority}
L. Lov\'asz, On decomposition of graphs, Studia Scientiarum Mathematicarum Hungarica
I(1-2) (1966) 237--238.

\bibitem{Przybylo22standard}
J. Przyby{\l}o, On the standard $(2,2)$-Conjecture, European J. Comb. 94 (2021) 103305. doi:10.1016/j.ejc.2020.103305

\end{thebibliography}
\end{document}